\documentclass[11pt]{amsart}

\usepackage{amsfonts, amsmath, amscd}
\usepackage[psamsfonts]{amssymb}

\usepackage{amssymb}

\usepackage{pb-diagram}

\usepackage[all,cmtip]{xy}

\usepackage[usenames]{color}

\newtheorem{theorem}{Theorem}[section]
\newtheorem{lemma}[theorem]{Lemma}
\newtheorem{corollary}[theorem]{Corollary}

\newtheorem{remark}[theorem]{Remark}
\newtheorem{proposition}[theorem]{Proposition}
\newtheorem{definition}[theorem]{Definition}

\newtheorem{fact}[theorem]{Fact}

\newtheorem{problem}[theorem]{Problem}

\numberwithin{equation}{section}

\newcommand{\CC}{C_k}
\newcommand{\NN}{\mathbb{N}}

\newcommand{\w}{\omega}

\newcommand{\e}{\varepsilon}

\renewcommand{\phi}{\varphi}

\newcommand{\supp}{\mathrm{supp}}

\input xy
\xyoption{all}

\title[Dunford--Pettis type properties for function spaces]{Dunford--Pettis type properties and the Grothendieck property for  function spaces}

\author{S. Gabriyelyan}
\address{Department of Mathematics, Ben-Gurion University of the Negev, Beer-Sheva, P.O. 653, Israel}
\email{saak@math.bgu.ac.il}

\author{J. K\c akol}
\address{ Faculty of Mathematics and Informatics. A. Mickiewicz University, 61-614 Pozna\'n, and Institute of Mathematics Czech Academy of Sciences, Prague}
\email{kakol@amu.edu.pl}


\thanks{Jerzy Kakol gratefully acknowledges the ﬁnancial support he received from the Center for Advanced Studies in Mathematics of the Ben Gurion University of the Negev during his visit April, 2018.}

\begin{document}

\begin{abstract}
For a Tychonoff space $X$, let $\CC(X)$ and $C_p(X)$ be the spaces of real-valued continuous functions $C(X)$ on $X$ endowed with the compact-open topology and the pointwise topology, respectively. If $X$ is compact, the classic result of A.~Grothendieck states that $\CC(X)$ has the Dunford-Pettis property   and the sequential Dunford--Pettis property. We extend Grothendieck's result by showing  that  $\CC(X)$ has both the Dunford-Pettis property   and the sequential Dunford-Pettis property if $X$ satisfies one of the following conditions: (i) $X$ is a hemicompact space, (ii) $X$ is a cosmic space (=a continuous image of a separable metrizable space), (iii) $X$ is the ordinal space $[0,\kappa)$ for some ordinal $\kappa$, or (vi) $X$ is a locally compact paracompact space. We show that if $X$ is a cosmic space, then $\CC(X)$ has the Grothendieck property if and only if every functionally bounded subset of $X$ is finite. We prove that $C_p(X)$ has the Dunford--Pettis property and the sequential Dunford-Pettis property for every Tychonoff space $X$, and $C_p(X) $ has the Grothendieck property if and only if every functionally bounded subset of $X$ is finite.
\end{abstract}

\maketitle

%%%%%%%%%%%%%%%%%%%%%%%%%%%
%%%%%%%%%%%%%%%%%%%%%%%%%%%
%%%%%%%%%%%%%%%%%%%%%%%%%%%
%%%%%%%%%%%%%%%%%%%%%%%%%%%

\section{Introduction}

%%%%%%%%%%%%%%%%%%%%%%%%%%%
%%%%%%%%%%%%%%%%%%%%%%%%%%%
%%%%%%%%%%%%%%%%%%%%%%%%%%%
%%%%%%%%%%%%%%%%%%%%%%%%%%%
The class of Banach spaces with the Dunford--Pettis property enjoying also   the Grothendieck property plays an essential and important role in the general theory of Banach spaces (particularly of continuous functions) and vector measures with several remarkable applications, we refer the reader to  \cite{Dales-Lau}, \cite{Diestel-DP}, \cite{Diestel} and   \cite[Chapter~VI]{Diestel-Uhl}.

It is well known by a result of A.~Grothendieck, see \cite[Corollary~4.5.10]{Dales-Lau}, that   for every injective compact space $K$, the Banach space $C(K)$ has the Grothendieck property. Consequently, this applies for each extremely disconnected compact space $K$. A.~Grothendieck also proved that the  Lebesgue spaces $L^{1}(\mu)$, the  spaces $\pounds^{\infty}$ and $\pounds^{1}$  have the Dunford--Pettis  property.

This line of research was continued by J.~Bourgain in \cite{bourgain}, where he showed that the spaces $C_{L^{1}}$ and $L^{1}_{C}$ enjoy also the Dunford--Pettis property. Moreover, in \cite{bourgain1} J.~Bourgain provided interesting sufficient conditions for subspaces $L$ of the Banach space $C(K)$ to have the Dunford--Pettis property. This results have been used by J.A.~Cima and R.M.~Timoney \cite{Ci-Ti} to study the Dunford--Pettis property for $T$-ivariant algebras  on $K$.
F.~Bombal and I.~Villanueva characterized in \cite{BomVil} those compact spaces $K$ such that $C(K)\hat{\otimes} C(K)$ have the Dunford--Pettis property.

Quite recently this area of research  around Dunford--Pettis property for Banach spaces   has been extended to  a more general setting  including general theory of locally convex spaces. This approach enabled specialists to apply this work to concrete problems related with the mean ergodic operators in Fr\'{e}chet spaces; we refer to articles \cite{ABR}, \cite{ABR1}, \cite{ABR2}, \cite{albaneze} and \cite{Bonet-Ricker} .

The classical Dunford--Pettis theorem states that for any measure $\mu$ and each Banach space $Y$, if $T:L_1(\mu)\to Y$ is a weakly compact linear operator, then $T$ is completely continuous (i.e., $T$ takes weakly compact sets in $L_1(\mu)$ onto norm compact sets in $Y$).
This result motivates Grothendieck to introduce the following property (for comments, see \cite[p.633-634]{Edwards}):
%Being motivated by results of A. Grothendieck the following definition of the Dunford--Pettis property was \textcolor[rgb]{1.00,0.00,0.00}{provided}  and studied in \cite[p.633-634]{Edwards}:
\begin{definition}[\cite{Grothen}] \label{def:DP}
A locally convex space $E$ is said to have  the {\em Dunford--Pettis property} ($(DP)$ {\em property} for short) if every continuous linear operator $T$ from $E$ into a quasi-complete locally convex space $F$, which transforms bounded sets of $E$ into relatively weakly compact subsets of $F$, also transforms absolutely convex weakly compact subsets of $E$ into relatively compact subsets of $F$.
\end{definition}
Actually, it suffices that $F$ runs over the class of Banach spaces, see \cite[p.633]{Edwards}.

A.~Grothendieck proved in \cite[Proposition~2]{Grothen} that a Banach space $E$ has the $(DP)$ property if and only if  given weakly null sequences $\{ x_n\}_{n\in\NN}$ and $\{ \chi_n\}_{n\in\NN}$ in $E$ and the Banach dual $E'$ of $E$, respectively, then $\lim_n \chi_n(x_n)=0$. He used this result to show that every Banach space $C(K)$ has the $(DP)$ property, see \cite[Th\'{e}or\`{e}me~1]{Grothen}. Extending this result to locally convex spaces (lcs, for short) and following \cite{Gabr-free-resp}, we  consider the following ``sequential'' version of the $(DP)$ property.

\begin{definition} \label{def:sDP}
A locally convex space $E$ is said to have  the {\em sequential Dunford--Pettis property} ($(sDP)$ {\em property})  if  given weakly null sequences $\{ x_n\}_{n\in\NN}$ and $\{ \chi_n\}_{n\in\NN}$ in $E$ and the strong dual $E'_\beta$ of $E$, respectively, then $\lim_n \chi_n(x_n)=0$.
\end{definition}

It turns out, as  A.A.~Albanese, J.~Bonet and W.J.~Ricker proved in \cite[Corollary~3.4]{ABR}, that the $(DP)$ property and the $(sDP)$ property coincide for the much wider class of Fr\'{e}chet spaces (or, even more generally, for  strict $(LF)$-spaces). In \cite[Proposition~3.3]{ABR} they showed  that every barrelled quasi-complete space with the  $(DP)$ property has also the  $(sDP)$ property.
For further results we refer the reader to \cite{ABR,Bonet-Lin-93,BFV,Diestel-DP} and reference therein.

For a Tychonoff (=completely regular and Hausdorff) space $X$, we denote by $\CC(X)$ and $C_p(X)$ the spaces of real-valued continuous functions $C(X)$ on $X$ endowed with the compact-open topology and the pointwise topology, respectively.
Being motivated by the aforementioned discussion and results it is natural to ask:

\begin{problem} \label{prob:Ck(X)-DPP}
Characterize Tychonoff spaces $X$ for which $\CC(X)$ and $C_p(X)$ have the $(DP)$ property or the $(sDP)$ property.
\end{problem}

A.~Grothendieck proved that the Banach space $C(\beta\NN)$, where $\beta\NN$ is the Stone--\v{C}ech compactification of the natural numbers $\NN$ endowed with the discrete topology, has the following property: Any weak-$\ast$ convergent sequence in the Banach dual of $C(\beta\NN)$ is also weakly convergent. This result motivates the following important property.
\begin{definition} \label{def:Grothendieck}
A locally convex space $E$ is said to have  the {\em Grothendieck property}  if every weak-$\ast$ convergent sequence in the strong dual $E'_\beta$ is weakly convergent.
\end{definition}
These results mentioned above motivate also the second general question considered in the paper.

\begin{problem} \label{prob:Ck(X)-Grothendieck}
Characterize Tychonoff spaces $X$ for which $\CC(X)$ and $C_p(X)$ have the Grothendieck property.
\end{problem}

For spaces $C_p(X)$ we obtain complete answers to Problems \ref{prob:Ck(X)-DPP} and  \ref{prob:Ck(X)-Grothendieck}.
Recall that a subset $A$ of a topological space $X$ is called {\em functionally bounded in $X$} if every $f\in C(X)$ is bounded on $A$.

\begin{theorem} \label{t:DP-Cp}
Let $X$ be a Tychonoff space. Then:
\begin{enumerate}
\item[{\rm (i)}] $C_p(X)$ has the $(sDP)$ property;
\item[{\rm (ii)}] $C_p(X)$ has the $(DP)$ property;
\item[{\rm (iii)}] $C_p(X) $ has the Grothendieck property if and only if every functionally bounded subset of $X$ is finite.
\end{enumerate}
\end{theorem}
We prove Theorem \ref{t:DP-Cp} in Section \ref{sec:Cp-GDP}. In this section we also recall some general results related to the $(DP)$ property and the Grothendieck property which are essentially used in the article.

For spaces $\CC(X)$ the situation is much more complicated. Following E.~Michael \cite{Mich}, a Tychonoff space $X$ is  a {\em cosmic space}  if $X$ is a continuous image of a separable metrizable space.
The following theorem is our second main result.
\begin{theorem} \label{t:DP-Ck}
Assume that a Tychonoff space  $X$ satisfies one of the following conditions:
\begin{enumerate}
\item[{\rm (i)}] $X$ is a hemicompact space;
\item[{\rm (ii)}] $X$ is a cosmic space;
\item[{\rm (iii)}] $X$ is the ordinal space $[0,\kappa)$ for some ordinal $\kappa$;
\item[{\rm (iv)}] $X$ is a locally compact paracompact space.
\end{enumerate}
Then $\CC(X)$ has the $(sDP)$ property and the $(DP)$ property.
\end{theorem}
Since  every compact subset of a cosmic space is metrizable, it is easy to see that indeed all four classes (i)-(iv) of Tychonoff spaces are independent (in the sense that there are spaces which belong to one of the classes but do not belong to other classes). In particular, the spaces  $\CC(\NN^\NN)$ and $\CC(\mathbb{Q})$ are  non-Fr\'{e}chet spaces with the $(DP)$ property and the $(sDP)$ property.

For the Grothendieck property, essentially using  Theorem \ref{t:DP-Cp} we prove the following result
(the definitions of $\mu$-spaces and sequential spaces are given before the proof of this result).

\begin{theorem} \label{t:Grothendieck-Ck}
Let $X$ be a Tychonoff space.
\begin{enumerate}
\item[{\rm (i)}] If $X$ is a $\mu$-space whose compact subsets are sequential (for example, $X$ is cosmic), then $\CC(X)$ has the Grothendieck property if and only if every functionally bounded subset of $X$ is finite.
\item[{\rm (ii)}]  If $X$ is a sequential space, then $\CC(X)$ has the Grothendieck property if and only if $X$ is discrete.
\end{enumerate}
\end{theorem}

Note that the condition on $X$ of being a sequential space in (ii) of Theorem \ref{t:Grothendieck-Ck} cannot be replaced by the condition ``$X$ is a $k$-space'',  as the compact space $\beta\NN$ shows.
We prove Theorems \ref{t:DP-Ck} and \ref{t:Grothendieck-Ck} in Section \ref{sec:Ck-GDP}.

% Stone--\v{C}ech compactifiacation $\beta\NN$ of $\NN$ shows.

%%%%%%%%%%%%%%%%%%%%%%%%%%%%%%%%%%%
%%%%%%%%%%%%%%%%%%%%%%%%%%%%%%%%%%%
%%%%%%%%%%%%%%%%%%%%%%%%%%%%%%%%%%%
%%%%%%%%%%%%%%%%%%%%%%%%%%%%%%%%%%%

\section{The $(DP)$ property and the Grothendieck property for $C_p(X)$} \label{sec:Cp-GDP}

%%%%%%%%%%%%%%%%%%%%%%%%%%%%%%%%%%%
%%%%%%%%%%%%%%%%%%%%%%%%%%%%%%%%%%%
%%%%%%%%%%%%%%%%%%%%%%%%%%%%%%%%%%%
%%%%%%%%%%%%%%%%%%%%%%%%%%%%%%%%%%%

In what follows we shall use the following result.
\begin{theorem}[\protect{\cite[Theorem~9.3.4]{Edwards}}] \label{t:Edwards-DP}
An lcs $E$ has the $(DP)$ property if and only if every absolutely convex, weakly compact subset of $E$ is precompact for the topology $\tau_{\Sigma'}$ of uniform convergence on the absolutely convex, equicontinuous, weakly compact subsets of $E'_\beta$.
\end{theorem}

A subset $A$ of a topological space $X$ is called {\em sequentially compact} if every sequence in $A$ contains a subsequence which is convergent in $A$. Following \cite{Gabr-free-resp}, a  Tychonoff space $X$ is {\em sequentially angelic} if a subset $K$ of $X$ is compact if and only if $K$ is sequentially compact. It is clear that every angelic space is sequentially angelic. An lcs $E$ is {\em weakly sequentially angelic} if the space $E_w$ is sequentially angelic, where $E_w$ denotes the space $E$ endowed with the weak topology. The following  proposition is proved in (ii) of Proposition~3.3 of \cite{ABR} (the condition of being quasi-complete is not used in the proof of this clause).

\begin{proposition}[\cite{ABR}] \label{p:DP-c0-quasibarrelled}
Assume that an lcs $E$ satisfies the following conditions:
\begin{enumerate}
\item[{\rm (i)}] $E$ has the $(sDP)$ property;
\item[{\rm (ii)}] both $E$ and $E'_\beta$ are weakly sequentially angelic.
\end{enumerate}
Then $E$ has the $(DP)$ property.
\end{proposition}

%\begin{proof}
%Suppose for a contradiction that $E$ does not have the $(DP)$ property. Then, by Theorem \ref{t:Edwards-DP}, there exists an absolutely convex, weakly compact subset $Q$ of $E$ which is not precompact in the topology $\tau_{\Sigma'}$. Therefore, by Theorem 5 of \cite{BGP}, there are an absolutely convex, equicontinuous, weakly compact subset $W$ of $E'_\beta$  and a sequence $\{ x_n: n\in\NN\}$ in $Q$ such that
%\begin{equation} \label{equ:DP-c0-quasibarrelled-1}
%x_n - x_m \not\in W^\circ \; \mbox{ for every distinct } n,m \in\NN.
%\end{equation}
%Since, by (ii), the compact set $Q$ is sequentially compact, passing to a subsequence if needed, we can  suppose that $x_n$ weakly converges to some element $x\in Q$. In particular, $x_n - x_{n+1}$ weakly converges to zero.
%Further, (\ref{equ:DP-c0-quasibarrelled-1}) implies that there is a sequence $S=\{ \mu_n: n\in\NN\}$ in $W$ such that
%\begin{equation} \label{equ:DP-c0-quasibarrelled-2}
%\mu_n (x_n - x_{n+1}) >1 \; \mbox{ for every  } n \in\NN.
%\end{equation}
%Since, by (ii), $W$ is sequentially compact, once again passing to a subsequence, we can additionally assume that $\mu_n$ weakly converges to some $\mu\in W$. Now (\ref{equ:DP-c0-quasibarrelled-2}) and (i) imply
%\[
%1< \mu_n (x_n - x_{n+1}) = (\mu_n -\mu) (x_n - x_{n+1}) + \mu (x_n - x_{n+1}) \to 0.
%\]
%This contradiction shows that $E$ has the $(DP)$ property.
%\end{proof}

Also we shall use repeatedly the next assertion, see \cite[Corollary~3.4]{ABR}.
\begin{proposition}[\cite{ABR}] \label{p:DP-LF-DF}
Let $E$ be a complete $(LF)$-space. Then $E$ has the $(DP)$ property if and only if it has the $(sDP)$ property.
\end{proposition}

\begin{proposition} \label{p:sDP-strong-dual-Schur}
Let $E$ be a locally complete lcs whose every separable bounded set is metrizable. Assume that $E$ is weakly sequentially angelic and  does not contain an isomorphic copy of $\ell_1$. Then $E$ has the $(sDP)$ property if and only if the strong dual $E'_\beta$ has the Schur property.
\end{proposition}

\begin{proof}
Assume that $E$ has the $(sDP)$ property. Let $\{ \chi_{n}\}_{n\in\NN}$ be a $\sigma(E',E'')$-null sequence in $E'$. Then Proposition 3.3(i) %\ref{p:Dunford-Pettis-Mackey}
of \cite{Gabr-free-resp} guarantees that $\chi_n \to 0$ in the Mackey topology $\mu(E',E)$ on $E'$. As $E$ does not contain an isomorphic copy of $\ell_1$, a result of Ruess \cite[Theorem~2.1]{ruess} %Proposition \ref{p:Dom-Drew}
asserts that $\chi_n \to 0$ in the strong topology. Thus $E'_\beta$ has the Schur property.
Conversely, if $E'_\beta$ has the Schur property, then $E$ has the $(sDP)$ property by Proposition 3.1(ii) % \ref{p:Schur-Dunford-Pettis}
of \cite{Gabr-free-resp}.
\end{proof}

If $E$ is a Fr\'{e}chet space, the necessity in the following corollary is Theorem 2.7 of \cite{albaneze}.
\begin{corollary} \label{c:Frechet-DP-L1-Schur}
Let $E$ be a strict  $(LF)$-space  not containing an isomorphic copy of $\ell_1$. Then $E$ has the $(DP)$ property if and only if the strong dual $E'_\beta$ has the Schur property.
\end{corollary}

\begin{proof}
The space $E$ is weakly angelic by a result of B.~Cascales and J.~Orihuela, see \cite[Proposition~11.3]{kak}. Also it is well known that $E$ is even complete whose every bounded set is metrizable, see \cite{bonet}. Taking into account that the $(DP)$ property is equivalent to the $(sDP)$ property in the class of strict $(LF)$-spaces by Proposition \ref{p:DP-LF-DF}, the assertion follows from Proposition \ref{p:sDP-strong-dual-Schur}.
\end{proof}

Let $X$ be a Tychonoff space. It is well known that the dual space of $C_p(X)$ is (algebraically) the linear space $L(X)$ of all linear combinations $\chi=a_1 x_1+\cdots + a_n x_n$, where $a_1,\dots,a_n$ are real numbers and $x_1,\dots,x_n\in X$. So
\[
\chi(f)=a_1 f(x_1)+\cdots + a_n f(x_n), \quad f\in C_p(X).
\]
If all the coefficients $a_1,\dots,a_n$ are nonzero and all $x_1,\dots,x_n$ are distinct, we set
\[
\supp(\chi):=\{ x_1,\dots,x_n\} \mbox{ and } \chi(x_i) := a_i, \; i=1,\dots,n.
\]
If $x\not\in \supp(\chi)$ we set $\chi(x):=0$.

We need the following proposition for which the statement  (i) is noticed on page 392 of \cite{FKS-feral} and  the case (ii) immediately follows from Theorem 5 or Theorem 10 of \cite{kakol}. Nevertheless, we provide  its complete and independent proof for the sake of completeness and reader convenience. Following \cite{FKS-feral},  an lcs $E$ is called {\em feral} if every infinite-dimensional subset of $E$ is unbounded. Recall that an lcs $E$ is called {\em $c_0$-barrelled} if every null sequence in the weak-$\ast$ dual of $E$ is equicontinuous, see \cite[Chapter~12]{Jar} or \cite[Chapter~8]{bonet}.

\begin{proposition} \label{p:Cp-strong-feral}
\begin{enumerate}
\item[{\rm (i)}]  The strong dual space $L_\beta(X)$ of $C_p(X)$ is feral.
\item[{\rm (ii)}] $C_p(X)$ is $c_0$-barrelled if and only if it is barrelled.
\end{enumerate}
\end{proposition}

\begin{proof}
(i) Suppose for a contradiction that there is a bounded infinite-dimensional subset $B$ in $L_\beta(X)$. For $n=1$, fix arbitrarily a nonzero $\chi_1\in B$ and let $x_1 \in \supp(\chi_1)$ be such that $\chi_1(x_1)=a_1 \not= 0$.
Since $B$ is infinite-dimensional, by induction,  for every natural number $n>1$ there exists a  $\chi_{n}\in B$ satisfying the following condition: there is an $x_{n}\in \supp(\chi_{n})$ such that
\begin{equation} \label{equ:DP-01}
x_{n} \not\in \bigcup_{i=1}^{n-1}\supp(\chi_i) \; \mbox{ and } \; \chi_{n}(x_{n})=a_{n} \not= 0.
\end{equation}
Clearly, all elements $x_n$ are distinct. Passing to a subsequence if needed we can assume that, for every $n\in\NN$, there is an open neighborhood $U_n$ of $x_n$ such that
\begin{equation} \label{equ:DP-02}
U_n \cap \left( \big[\supp(\chi_n)\setminus\{ x_n\}\big] \cup \bigcup_{i=1}^{n-1} U_i \right) =\emptyset.
\end{equation}
Finally, for every $n\in\NN$, take a function $f_n\in C(X)$ such that $\supp(f_n)\subseteq U_n$ and $f_n(x_n)=n/a_n$. It is easy to see $f_n \to 0$ in $C_p(X)$, and hence the sequence $S=\{ f_n: n\in\NN\}$ is bounded in $C_p(X)$. The choice of $f_n$, (\ref{equ:DP-01}) and (\ref{equ:DP-02}) imply $\chi_n(f_n)=f_n(x_n)=n \to\infty$. Therefore the sequence $\{ \chi_n: n\in\NN\}\subseteq B$ is unbounded, a contradiction.

\smallskip
(ii) If $C_p(X)$ is barrelled then trivially it is $c_0$-barrelled. Conversely, assume that $C_p(X)$ is $c_0$-barrelled.  By the Buchwalter--Schmets theorem, it suffices to prove that every functionally bounded subset of $X$ is finite. Suppose for a contradiction that $X$ has a one-to-one sequence $\{ a_n:n \in\NN\}\subseteq X$  which is functionally bounded in $X$. For every $n\in \NN$, set $\chi_n= 2^{-n}\delta_{a_n}\in L(X)$, where $\delta_{a_n}$ is the Dirac measure at $a_n$. Then, for every $f\in C(X)$, we have
\[
|\chi_n (f)| \leq 2^{-n}\cdot \sup\{ |f(a_n)|: n\in\NN\} \to 0.
\]
Therefore $\chi_n $ is a weak-$\ast$ null sequence. Since $C_p(X)$ is $c_0$-barreled we obtain that the sequence $S=\{ \chi_n: n\in\NN\}$ is equicontinuous. So there is a neighborhood $U$ of zero in $E$ such that $S\subseteq U^\circ$. Since, by the Alaoglu theorem, $U^\circ$ is a $\sigma(E',E)$-compact convex subset of $E'$,  we obtain that $S$ is strongly bounded, see Theorem  11.11.5 of \cite{NaB}.  Clearly, the sequence $S$ is infinite-dimensional and hence, by (i),  $S$ is not  bounded in the strong topology. This contradiction finishes the proof.
\end{proof}

Now Theorem \ref{t:DP-Cp}  follows immediately from Proposition \ref{p:Cp-strong-feral} and the next result below.

\begin{theorem} \label{t:Cp-not-DP}
Let $E$ be an lcs whose strong dual is  feral. Then:
\begin{enumerate}
\item[{\rm (i)}] $E$ has the $(sDP)$ property;
\item[{\rm (ii)}] $E$ has the $(DP)$ property;
\item[{\rm (iii)}] $E$ has the Grothendieck property if and only if it is $c_0$-barrelled.
\end{enumerate}
\end{theorem}

\begin{proof}
(i) Let $S'=\{ \chi_n :n\in\NN\}$ be a weakly null sequence in  $E'_\beta$. As $E'_\beta$ is feral, the sequence  $S'$ is finite-dimensional, and hence for every weakly null (even bounded) sequence $\{ x_n :n\in\NN\}$ in $E$ we trivially have $\chi_n (x_n) \to 0$. Thus $E$ has the $(sDP)$ property.

\smallskip
(ii) We use Theorem \ref{t:Edwards-DP}. First we note that every weakly compact subset of $E'_\beta$ is finite-dimensional. Therefore, every polar $A^\circ$ of an absolutely convex, equicontinuous and weakly compact subset $A$ of $E'_\beta$ defines a weak neighborhood at zero of $E$. Therefore $\tau_{\Sigma'}$ coincides with the weak topology  of $E$. Thus $E$ has the $(DP)$ property.

\smallskip
(iii) Assume that $E$ has the Grothendieck property. Suppose for a contradiction that $E$ is not $c_0$-barrelled. Then there exists a weak-$\ast$ null sequence $S$ in $E'$ which is not equicontinuous. Clearly, $S$ is infinite-dimensional. Since $E'_\beta$ is feral it follows that $S$ is not strongly bounded. Thus $S$ does not converge to zero in the weak topology of $E'_\beta$, and hence $E$ does not have the Grothendieck property. This contradiction shows that $E$ must be $c_0$-barrelled.

Conversely, assume that $E$ is $c_0$-barrelled and let $S=\{ \chi_n: n\in\NN\}$ be a weak-$\ast$ null sequence in $E'$. Then $S$ is equicontinuous. So there is a neighborhood $U$ of zero in $E$ such that $S\subseteq U^\circ$. Since, by the Alaoglu theorem, $U^\circ$ is a weak-$\ast$ compact convex subset of $E'$,  we obtain that $S$ is strongly bounded, see Theorem  11.11.5 of \cite{NaB}. Therefore $S$ is finite-dimensional because $E'_\beta$ is feral. Hence $S$ is also a weakly null sequence in $E'_\beta$. Thus $E$  has the Grothendieck property.
\end{proof}

%%%%%%%%%%%%%%%%%%%%%%%%%%%%%%%%%%%%%%%%%%
%%%%%%%%%%%%%%%%%%%%%%%%%%%%%%%%%%%%%%%%%%
%%%%%%%%%%%%%%%%%%%%%%%%%%%%%%%%%%%%%%%%%%
%%%%%%%%%%%%%%%%%%%%%%%%%%%%%%%%%%%%%%%%%%
%%%%%%%%%%%%%%%%%%%%%%%%%%%%%%%%%%%%%%%%%%

\section{ The $(DP)$ property and the Grothendieck property for $\CC(X)$} \label{sec:Ck-GDP}

%%%%%%%%%%%%%%%%%%%%%%%%%%%%%%%%%%%%%%%%%%
%%%%%%%%%%%%%%%%%%%%%%%%%%%%%%%%%%%%%%%%%%
%%%%%%%%%%%%%%%%%%%%%%%%%%%%%%%%%%%%%%%%%%
%%%%%%%%%%%%%%%%%%%%%%%%%%%%%%%%%%%%%%%%%%
%%%%%%%%%%%%%%%%%%%%%%%%%%%%%%%%%%%%%%%%%%

%In the next two theorems we obtain two sufficient conditions on Tychonoff spaces $X$ for which $\CC(X)$ has the $(sDP)$ property.

For a Tychonoff space $X$, we denote by $M_c(X)$ the space of all finite real regular Borel measures on $X$ with compact support (which will be denoted by $\mu,\nu$ etc.). It is well known that $M_c(X)$ is the dual space of $\CC(X)$.

Let $K$ be a compact subspace of a Tychonoff space $X$. Denote by $M_K(X)$ the linear subspace of $M_c(X)$ of all measures with support in $K$.
Denote by $J: M(K)\to M_K(X)$ the natural inclusion map defined by
\[
J(\nu)(A):= \nu(A\cap K),
\]
where $A$ is a Borel subset of $X$.

\begin{lemma} \label{l:M(K)-M(X)}
Let $K$ be a compact subspace of a Tychonoff space $X$. Then $J$ is a linear isomorphism of the Banach space $M(K)$ onto the subspace $M_K(X)$ of $M_c(X)_\beta$.
\end{lemma}

\begin{proof}
It is clear that $J$ is a linear isomorphism. We show that $J$ is a homeomorphism.

Denote by $S$ the restriction map from $\CC(X)$ to $C(K)$, i.e., $S(f):= f|_K$ for every $f\in\CC(X)$. Clearly, $S$ is a continuous linear operator. Therefore its adjoint map $S^\ast: M(K) \to M_c(X)_\beta$ is continuous, see \cite[Theorem~8.11.3]{NaB}. Noting that
\[
S^\ast (\nu)(f)=\nu(S(f)), \quad \nu\in M(K), \; f\in \CC(X),
\]
we see that $J$ is a corestriction of $S^\ast$ to $M_K(X)$. Thus $J$ is continuous.

To show that $J$ is also open it is sufficient to prove that $J(B_{M(K)})$ contains a neighborhood of zero in $M_K(X)$, where $B_{M(K)}$ is the closed unit ball of the Banach space $M(K)$.  Define
\[
B:= \{ f\in C(X): \; |f(x)|\leq 1 \mbox{ for every } x\in X\}.
\]
It is clear that $B$ is a bounded subset of $\CC(X)$. Therefore, $B^\circ \cap M_K(X)$ is a neighborhood of zero in $M_K(X)$. We show that $B^\circ \cap M_K(X) \subseteq J(B_{M(K)})$. Indeed, let $\mu\in B^\circ \cap M_K(X)$ and denote by $\nu$ the restriction of $\mu$ onto $K$; so $\nu\in M(K)$ and $J(\nu)=\mu$. We have to prove that $\nu\in B_{M(K)}$. Fix an arbitrary function $g\in B_{C(K)}$. By the Tietze--Urysohn theorem, choose an extension $\widetilde{g}\in C(X)$  of $g$ onto $X$ such that $|\widetilde{g}(x)|\leq 1$ for every $x\in X$. Then $\widetilde{g}\in B$, and since $\mu\in B^\circ$ we obtain
\[
|\nu(g)|=\left| \int_K g(x) d \nu\right| =\left| \int_X \widetilde{g}(x)d \mu\right|  \leq 1.
\]
Thus $\nu\in B_{M(K)}$.
\end{proof}

Below we provide  a quite general condition on a Tychonoff space $X$ for which the space $\CC(X)$ has the $(sDP)$ property.
Recall that the sets
\[
[K;\e] :=\{ f\in C(X): |f(x)|<\e \; \forall x\in K\},
\]
where $K$ is a compact subset of  $X$ and $\e>0$, form a base at zero of the compact-open topology $\tau_k$ of $\CC(X)$.

\begin{theorem} \label{t:sequential-DP-c0-quasibarrelled}
Each $c_0$-barrelled space $\CC(X)$ has the $(sDP)$ property.
\end{theorem}

\begin{proof}
Let $\{ f_n\}_{n\in\NN}$ and $\{ \mu_n\}_{n\in\NN}$ be  weakly null sequences in $\CC(X)$ and its strong dual $M_c(X)_\beta$, respectively. We have to show that $\lim_n \mu_n(f_n)=0$.

Observe that the weak topology of $M_c(X)_\beta$ is stronger than the weak-$\ast$ topology on $M_c(X)$.
Therefore the $c_0$-barrelledness of  $\CC(X)$ implies that the sequence $S=\{ \mu_n\}_{n\in\NN}$ is equicontinuous. So there is a compact subset $K$ of $X$ and $\e>0$ such that $S\subseteq [K;\e]^\circ$. Since $X$ is Tychonoff, it follows that $\supp(\mu_n)\subseteq K$ for every $n\in\NN$. Indeed, otherwise, there is a function $f\in C(X)$ with support in $X\setminus K$ such that $\mu(f)>0$. It is clear that $\lambda f\in [K;\e]$ for every $\lambda>0$, and hence $\mu(\lambda f)> 1$ for sufficient large $\lambda$, a contradiction.

For every $n\in\NN$, denote by $\nu_n$ the restriction of $\mu_n$ onto $K$, i.e., $\nu_n(A):=\mu_n(A\cap K)$ for every Borel subset $A$ of $X$. By Lemma \ref{l:M(K)-M(X)}, $\nu_n\to 0$ in the weak topology of the Banach space $M(K)$.
Observe that the sequence $\{ f_n|_K\}_{n\in\NN}$ is weakly null in the Banach space $C(K)$ because the restriction map $S:\CC(X) \to C(K)$, $S(f):= f|_K$, is continuous and hence is weakly continuous. Since the support of $\mu_n$ is contained in $K$ we obtain
\[
\nu_n(f|_K)=\int_K f|_K(x)d\nu_n = \int_X f(x)d\mu_n =\mu_n(f)
\]
for every $f\in C(X)$. Now this equality  and the $(sDP)$ property of $C(K)$ imply
$
\lim_n \mu_n(f_n)=\lim_n \nu_n\big(f_n|_K\big)=0.
$
Thus $\CC(X)$ has the $(sDP)$ property.
\end{proof}

Let $\alpha$ and $\kappa$ be ordinals such that $\alpha<\kappa$.
Since $[0,\alpha]$ and $(\alpha,\kappa)$ are clopen subspaces of $[0,\kappa)$, we have
\begin{equation} \label{equ:DP-1}
\CC\big([0,\kappa)\big) = C\big([0,\alpha]\big) \oplus \CC\big((\alpha,\kappa)\big),
\end{equation}
and hence, for strong dual spaces,
\begin{equation} \label{equ:DP-2}
M_c\big([0,\kappa)\big)_\beta =M_c\big([0,\alpha]\big)_\beta\oplus M_c\big((\alpha,\kappa)\big)_\beta.
\end{equation}

\begin{proposition} \label{p:Ck-ordinal-c0-barrelled}
For every ordinal $\kappa$, the space $\CC \big([0,\kappa)\big)$ is $c_0$-barrelled.
\end{proposition}

\begin{proof}
If $\kappa$ is  a successor ordinal or has countable cofinality, then $\CC \big([0,\kappa)\big)$ is a Banach space or a Fr\'{e}chet space, respectively. Therefore $\CC \big([0,\kappa)\big)$ is even a barrelled space. Assume now that the cofinality $\mathrm{cf}(\kappa)$ of $\kappa$ is uncountable.  For simplicity, set $E:=\CC\big([0,\kappa)\big)$ and let $E'_\beta :=M_c\big([0,\kappa)\big)_\beta $ be the strong dual of $E$.

Let $A=\{ \mu_n\}_{n\in\NN}$ be a weakly-$\ast$ null sequence in $E'_\beta$. For every $n\in \NN$, the support of $\mu_n$ is compact and hence there is an ordinal $\alpha_n$, $\alpha_n< \kappa$, such that  $\supp(\mu_n) \subseteq [0,\alpha_n]$. Set $\alpha:= \sup\{ \alpha_n: n\in\NN\}$. Since $\mathrm{cf}(\kappa)>\w$, we have $\alpha<\kappa$. For every $n\in \NN$, denote by $\nu_n $ the restriction $ \mu_n |_{[0,\alpha]}$ of $\mu_n$ onto $[0,\alpha]$. Since the restriction map $T: E\to C\big([0,\alpha]\big), T(f)=f|_{[0,\alpha]}$, is surjective  we obtain that the sequence $S=\{ \nu_n\}_{n\in\NN}$ is a weakly-$\ast$ null sequence in the dual $M\big([0,\alpha]\big)$ of the Banach space $C\big([0,\alpha]\big)$. Thus $S$ is equicontinuous, and hence there is $\lambda>0$ such that
\[
S \subseteq \lambda \tilde{B}_\alpha^\circ, \mbox{ where } \tilde{B}_\alpha :=\big\{ g\in C\big([0,\alpha]\big): |g(x)|\leq 1 \mbox{ for all } x\in [0,\alpha]\big\}.
\]
Set $B_\alpha:=  \tilde{B}_\alpha \times \CC\big((\alpha,\kappa)\big)$. It follows from (\ref{equ:DP-1}) that $B_\alpha$ is a neighborhood of zero in $E$.
Then, for every $f\in B_\alpha$ and each $\mu_n\in A$, we have
\[
|\mu_n(f)|=\left| \int_{[0,\kappa)} f(x) d\mu_n \right| =\left| \int_{[0,\alpha]} f|_{[0,\alpha]} (x) d\nu_n \right|  \leq \lambda.
\]
 Therefore $A\subseteq \lambda B_\alpha^\circ$ and hence  $A$ is equicontinuous. Thus $E$ is $c_0$-barrelled.
\end{proof}
The Nachbin--Shirota theorem, see \cite{bonet},  implies that $\CC \big([0,\kappa)\big)$ is barrelled if and only if $\kappa$ is  a successor ordinal or has countable cofinality. Therefore, if  the cofinality $\mathrm{cf}(\kappa)$ of $\kappa$ is uncountable (for example, $\kappa=\w_1$), then $\CC \big([0,\kappa)\big)$ is $c_0$-barrelled but not barrelled.

Below we prove Theorem \ref{t:DP-Ck}.

\medskip
{\em Proof of Theorem \ref{t:DP-Ck}}.
(i) Assume that $X$ is a hemicompact space.
Then the space $\CC(X)$ has the $(sDP)$ property by Theorem \ref{t:sequential-DP-c0-quasibarrelled}. Applying
Proposition \ref{p:DP-LF-DF} we obtain that the space $\CC(X)$ has also the $(DP)$-property.

\smallskip
(ii) Assume that $X$ is a cosmic space. First we recall that each cosmic space  is Lindel\"{o}f  and every its compact subset  is metrizable, see \cite{Mich}. Therefore $X$ is a $\mu$-space, and hence $\CC(X)$ is barrelled. Proposition 10.5 of  \cite{Mich} implies that $C_p(X)$ is a cosmic space, and therefore every compact subset of $C_p(X)$ is metrizable. Observe that the weak topology of $\CC(X)$ is finer than the pointwise topology of $C_p(X)$. Thus  every weakly compact subset of $\CC(X)$ is metrizable, and hence is $\CC(X)$ is weakly sequentially angelic.

The space $\CC(X)$ has the $(sDP)$ property by Theorem \ref{t:sequential-DP-c0-quasibarrelled}.
To prove that $\CC(X)$ has also the $(DP)$ property  we apply Proposition \ref{p:DP-c0-quasibarrelled}. We proved that $\CC(X)$ is barrelled and weakly sequentially angelic. Therefore it remains to check that the strong dual $M_c(X)_\beta$ of $\CC(X)$ is weakly sequentially angelic.
Observe that the space $C_p(X)$ being cosmic is separable, see \cite[p.994]{Mich}. Therefore, by Corollary 4.2.2 of \cite{mcoy}, also the space $\CC(X)$  is separable. It follows that  $M_c(X)$ with the weak-$\ast$ topology admits a weaker metrizable locally convex vector topology. As the weak topology of the strong dual $M_c(X)_\beta$ of $\CC(X)$ is evidently stronger than the weak-$\ast$ topology on $M_c(X)$, we obtain that every weakly compact subsets of $M_c(X)_\beta$ is even metrizable, and therefore $M_c(X)_\beta$ is weakly sequentially angelic. Finally, Proposition \ref{p:DP-c0-quasibarrelled} implies that $\CC(X)$ has the $(DP)$ property.

\smallskip
(iii) Let $X=[0,\kappa)$ for some ordinal $\kappa$.  If $\kappa$ is  a successor ordinal or has countable cofinality, then  $[0,\kappa)$ is hemicompact. Thus, by (i), $\CC \big([0,\kappa)\big)$ has the $(sDP)$ property and the $(DP)$ property. Assume now that the cofinality $\mathrm{cf}(\kappa)$ of $\kappa$ is uncountable. %For simplicity, set $E:=\CC\big([0,\kappa)\big)$ and let $E'_\beta :=M_c\big([0,\kappa)\big)_\beta $ be the strong dual of $E$.

Proposition \ref{p:Ck-ordinal-c0-barrelled} implies that $\CC\big([0,\kappa)\big)$ is $c_0$-barrelled. Thus, by Theorem \ref{t:sequential-DP-c0-quasibarrelled}, $\CC\big([0,\kappa)\big)$ has the $(sDP)$ property. We show below that $\CC\big([0,\kappa)\big)$ also has the $(DP)$ property.

Suppose for a contradiction that $\CC\big([0,\kappa)\big)$ does not have the $(DP)$ property. Then, by Theorem \ref{t:Edwards-DP}, there exists an absolutely convex, weakly compact subset $Q$ of $\CC\big([0,\kappa)\big)$ which is not precompact in the topology $\tau_{\Sigma'}$. Therefore, by Theorem 5 of \cite{BGP}, there are an absolutely convex, equicontinuous, weakly compact subset $W$ of $M_c\big([0,\kappa)\big)_\beta$  and a sequence $\{ f_n: n\in\NN\}$ in $Q$ such that
\begin{equation} \label{equ:DP-c0-quasibarrelled-3}
f_n - f_m \not\in W^\circ \; \mbox{ for every distinct } n,m \in\NN.
\end{equation}
For every $n\in \NN$, choose $\alpha_n< \kappa$ such that $f_n(x)=f_n(\alpha_n)$ for every $x>\alpha_n$, see \cite[Example~3.1.27]{Eng}. Set $\alpha:= \sup\{ \alpha_n: n\in\NN\}$. Since $\mathrm{cf}(\kappa)>\w$, we have $\alpha<\kappa$. For every $n\in \NN$, set $g_n:= f_n|_{[0,\alpha]}$. Since the restriction operator $T: \CC\big([0,\kappa)\big)\to C\big([0,\alpha]\big)$ onto $[0,\alpha]$ is continuous, it is weakly continuous and hence $T(Q)$ is a weakly compact subset of the Banach space $C\big([0,\alpha]\big)$. By the Eberlein--\v{S}mulian theorem, $T(Q)$  is weakly sequentially compact. Therefore, passing to a subsequence if needed, we can  assume that $g_n$ weakly converges to some $g\in C\big([0,\alpha]\big)$. In particular, $g_n(\alpha)=f_n(\alpha) \to g(\alpha)$. Define $f\in E$ by $f|_{[0,\alpha]}:= g$ and $f|_{(\alpha,\kappa)}:=g(\alpha)$. Taking into account that $f_n|_{(\alpha,\kappa)}=f_n(\alpha) \mathbf{1}_{(\alpha,\kappa)}$ (where $\mathbf{1}_A$ denotes the characteristic function of a subset $A$), we obtain  that $f_n\to f$  in the weak topology of $E$. Note that $f_n - f_{n+1}$ weakly converges to zero.

Now (\ref{equ:DP-c0-quasibarrelled-3}) implies that there is a sequence $S=\{ \mu_n: n\in\NN\}$ in $W$ such that
\begin{equation} \label{equ:DP-c0-quasibarrelled-4}
\mu_n (f_n - f_{n+1}) >1 \; \mbox{ for every  } n \in\NN.
\end{equation}
For every $n\in \NN$, choose $\alpha\leq\beta_n< \kappa$ such that $\supp(\mu_n) \subseteq [0,\beta_n]$. Set $\beta:= \sup\{ \beta_n: n\in\NN\}$. Since $\mathrm{cf}(\kappa)>\w$, we have $\beta<\kappa$.
By Lemma \ref{l:M(K)-M(X)}, the sequence $S$ is contained in the closed subspace $L:= M_{[0,\beta]} \big([0,\kappa)\big)$ of $M_c \big([0,\kappa)\big)_\beta$ and $L$ is topologically isomorphic to the Banach space $M([0,\beta])$. Therefore, the set $W_\beta := W\cap L$ is a weakly compact subset of $M([0,\beta])$. Once again applying the Eberlein--\v{S}mulian theorem and passing to a subsequence if needed, we can assume that $\{ \mu_n\}_{n\in\NN}$ weakly converges to a measure $\mu\in L\subseteq M_c \big([0,\kappa)\big)_\beta$.
Now (\ref{equ:DP-c0-quasibarrelled-4}) and the $(sDP)$ property of $\CC\big([0,\kappa)\big)$ proved above imply
\[
1< \mu_n (f_n - f_{n+1}) = (\mu_n -\mu) (f_n - f_{n+1}) + \mu (f_n - f_{n+1}) \to 0.
\]
This contradiction shows that $\CC(X)$ has the $(DP)$ property.

\smallskip
(iv) Assume that $X$ is a locally compact and paracompact space. Theorem 5.1.27 of \cite{Eng} states that $X=\bigoplus_{i\in I} X_i$ is the direct topological sum of a family $\{ X_i\}_{i\in I}$ of Lindel\"{o}f locally compact spaces. Since all $X_i$ are hemicompact by \cite[Ex.3.8.c]{Eng}, (i) implies that all spaces $\CC(X_i)$ have the $(DP)$ property and the $(sDP)$ property. Therefore the space $\CC(X)=\prod_{i\in I} \CC(X_i)$ has the $(DP)$ property by \cite[9.4.3(a)]{Edwards}. Now we check that $\CC(X)$ has also the $(sDP)$ property.

Let $\{ f_n\}_{n\in\NN}$ and $\{ \mu_n\}_{n\in\NN}$ be weakly null sequences in $\CC(X)$ and $M_c(X)_\beta$, respectively. Choose a countable subfamily $J$ of the index set $I$ such that
\[
\bigcup_{n\in\NN} \supp(\mu_n) \subseteq \bigcup_{j\in J} X_j, \; \mbox{ and set } \; Y:=\bigcup_{j\in J} X_j.
\]
For every $n\in\NN$, set $g_n:= f_n |_Y$ and let $\nu_n$ be the restriction of $\mu_n$ onto $Y$. By construction,  $Y$ and $X\setminus Y$ are clopen subsets of $X$. Therefore $\CC(X)=\CC(Y)\times \CC(X\setminus Y)$ and hence  $M_c(X)_\beta = M_c(Y)_\beta \times M_c(X\setminus Y)_\beta$.  Since the projection onto the first summand is continuous, it is weakly continuous as well. Thus $ \{ g_n\}_{n\in\NN}$ and  $ \{ \nu_n\}_{n\in\NN}$  are weakly null sequences in $\CC(Y)$ and $M_c(Y)_\beta$, respectively. As $Y$ is hemicompact, (i) implies  $\mu_n(f_n)= \nu_n(g_n) \to 0$ as $n\to\infty$. Thus $\CC(X)$ has the $(sDP)$ property.
\qed

\medskip

Having in mind the $(DP)$ property one may ask: {\em Under which condition a null-sequence $\{ f_n: n\in\NN\}$  in $C_{p}(X)$ weakly converges to zero in $\CC(X)$}?  Recall  the following known observation which is a consequence of  the Lebesgue's dominated convergence theorem and the fact that every measure $\mu \in \CC(X)'=M_c(X)$ has compact support:
\begin{fact}\label{fact:weakly-null-Ck}
Let $X$ be a Tychonoff space and let $S=\{ f_n:n\in\NN\}$ be a null-sequence in $C_p(X)$. If $S$ is bounded in $\CC(X)$, then $f_{n}\rightarrow 0$ in the space $\CC(X)_{w}$, which means the space $\CC(X)$ endowed with the weak topology of $\CC(X)$.
\end{fact}

Using (i) of Theorem \ref{t:DP-Ck} and Fact \ref{fact:weakly-null-Ck} we obtain the following easy
\begin{proposition} \label{p:bounded-sequence-in-Ck}
Let $X$ be a hemicompact space and let $S=\{ f_n:n\in\NN\}$ be a null-sequence in $C_p(X)$. Then the following assertions are equivalent:
\begin{enumerate}
\item[{\rm (i)}] $S$ is bounded in the space $\CC(X)$;
\item[{\rm (ii)}] $\mu_n(f_n)\to 0$ for every weakly null-sequence $\{\mu_{n}: n\in\NN\}$ in the strong dual of $\CC(X)$.
\end{enumerate}
\end{proposition}

\begin{proof}
(i)$\Rightarrow$(ii) immediately follows from (i) of Theorem \ref{t:DP-Ck} and Fact \ref{fact:weakly-null-Ck}.

(ii)$\Rightarrow$(i) Let $\{ K_n:n\in\NN\}$ be a fundamental  (increasing) sequence of compact sets in $X$. If all the $K_n$ are finite, then $C_{p}(X)=\CC(X)$ and hence $S$ is bounded by the assumption $f_{n}\rightarrow 0$ in $C_{p}(X)$. So we shall assume that all $K_n$ are infinite.

Suppose for a contradiction that $S$ is not bounded in $\CC(X)$. Then there exists $K:=K_{m}$ such that $S$ is unbounded in the Banach space $C(K):=(C(K),\|.\|)$. Let $k_1<k_2<\cdots$ be a sequence in $\NN$ such that $\|f_{k_n}\|\geq n$ for all $n\in \NN$. For every $n\in\NN$, pick $x_{n}\in K$ such that $|f_{k_n}(x_{n})|=\|f_{k_n}\|$ and set
\[
\mu_{i}:= \|f_{k_n}\|^{-1}\delta_{x_{n}} \mbox{ if } i=k_n \mbox{ for some } n\in\NN, \mbox{ and } \mu_{i}:=0 \mbox{ otherwise}.
\]
Then the sequence $M:=\{ \mu_{i}:i\in\NN\}$ converges to zero in the norm dual of $C(K)$.  It follows from Lemma \ref{l:M(K)-M(X)} that
%Since the restriction map $\pi: \CC(X)\rightarrow C(K)$ is a continuous surjection by the Tietze--Urysohn extension theorem and the adjoint map $\pi^\ast: C(K)' \to \CC(X)'$ is strongly continuous, we obtain that
$\mu_i \to 0$ in the strong dual of $\CC(X)$. Therefore $M$ is a weakly null-sequence in the strong dual of $\CC(X)$. But since $|\mu_{k_n}(f_{k_n})|=1$ for every $n\in M$, we obtain that (ii) does not hold. This contradiction shows that $S$ is bounded in $\CC(X)$.
\end{proof}

\begin{remark} \label{r:Ck-Cp-bounded} {\em
Let $X$ be a Tychonoff space containing an infinite compact subset $K$. Then $C_p(X)$ contains a null-sequence which is not bounded for $\CC(X)$.
Indeed, since $K$ is infinite, there is an infinite discrete sequence $\{ x_n\}_{n\in\NN}$ in $K$ with pairwise disjoint neighborhoods $V_n$ of $x_n$ in $X$, see \cite[Lemma 11.7.1]{Jar}. For every $n\in\NN$, choose a function $f_n: X \to [0,n]$ with support in $V_n$ and $f_n(x_n)=n$. As $V_n$ are pairwise disjoint, $f_n \to 0$ in $C_p(X)$. It is clear that, for each $m\in\NN$, we have
\[
f_n \not\in\{ f\in\CC(X): |f(x)|\leq m \mbox{ for all } x\in K\}, \mbox{ for every } n>m.
\]
Thus $\{ f_{n}:n\in\NN\}$ is not bounded in $\CC(X)$.} \qed
\end{remark}

To prove Theorem \ref{t:Grothendieck-Ck} we need the following lemma which
 (for compact spaces) actually is noticed in \cite[p.138]{Dales-Lau}. Recall that a sequence $\{ x_n\}_{n\in\NN}$ in a topological space $X$ is called {\em trivial} if there is $m\in\NN$ such that $x_n =x_m$ for every $n\geq m$.
%For compact spaces $K$, the next lemma is noticed in \cite[p.138]{Dales-Lau}.
\begin{lemma} \label{l:Grothendieck-sequence}
Let $X$ be a Tychonoff space. If $\CC(X)$ has the Grothendieck property, then $X$ does not contain non-trivial convergent sequences.
\end{lemma}

\begin{proof}
Suppose for a contradiction that there is a sequence $S=\{ x_n\}_{n\in\NN}$ converging to a point $x\in X\setminus S$, we shall assume that $S$ is one-to-one. Then the sequence $\{ \delta_{x_n}-\delta_x\}_{n\in\NN}$ evidently converges weakly-$\ast$ to $0$ (as usual $\delta_z$ denotes the Dirac measure at $z\in X$). Set $K:=S\cup\{ x\}$, so $K$ is a compact subset of $X$. It is easy to see that the map
\[
\mu \mapsto \sum_{n\in\NN} \mu\big( \{ x_n\} \big), \quad \mu\in M(K),
\]
is a continuous linear functional of the Banach space $M(K)$. Then, by Lemma \ref{l:M(K)-M(X)} and the Hahn--Banach extension theorem, there exists an extension  $\chi$ of this map to a continuous linear functional on $M_c(X)_\beta$. Since $\chi\big( \delta_{x_n}-\delta_x\big) =1$ for every $n\in\NN$, we see that $\delta_{x_n}-\delta_x \not\to 0$ in the weak topology of $M_c(X)_\beta$.
\end{proof}

For the convenience of the reader we recall some definitions used in Theorem \ref{t:Grothendieck-Ck}.
%A subset $A$ of a topological space $X$ is called {\em functionally bounded in $X$} if every $f\in C(X)$ is bounded on $A$; the space
A  topological space $X$ is a {\em $\mu$-space} if $X$ is Tychonoff  and every functionally bounded subset of $X$ is relatively compact. The Nachbin--Shirota theorem states that a Tychonoff  space $X$ is a $\mu$-space if and only if $\CC(X)$ is barrelled.
A topological space $X$ is called {\em sequential} if for each non-closed subset $A\subseteq X$ there is a sequence $\{a_n\}_{n\in\w}\subseteq A$ converging to some point $a\in \overline{A}\setminus A$.
We note that a sequential space $X$ is discrete if and only if it does not contain non-trivial convergent sequences. (Indeed, if $z\in X$ is non-isolated, then the set $A:=X\setminus\{z \}$  is non-closed and hence there is a sequence (necessarily non-trivial) in $A$ converging to $z$.)
Now we are ready to prove Theorem \ref{t:Grothendieck-Ck}.

\medskip
{\em Proof of Theorem \ref{t:Grothendieck-Ck}}.
(i) Let $X$ be a $\mu$-space whose compact subsets are sequential. Assume that $\CC(X)$ has the Grothendieck property. Let $A$ be a functionally bounded subset of $X$. Then its closure $\overline{A}$ is a compact subset of $X$.  By Lemma \ref{l:Grothendieck-sequence}, $\overline{A}$ does not contain non-trivial convergent sequences. Since $\overline{A}$ is sequential we obtain that $A$ is finite. Conversely, if every functionally bounded subset of $X$ is finite, then $\CC(X)=C_p(X)$ and Theorem \ref{t:DP-Cp} applies.

(ii) Let $X$ be a sequential space. Assume that $\CC(X)$ has the Grothendieck property. Then, by Lemma \ref{l:Grothendieck-sequence}, the space $X$ does not  contain non-trivial convergent sequences. Since $X$ is sequential we obtain that $X$ is discrete. Conversely, if $X$ is discrete, then $\CC(X)=C_p(X)$ and Theorem \ref{t:DP-Cp} applies. \qed

\medskip

\bibliographystyle{amsplain}

\end{document}